\documentclass[12pt,oneside]{amsart}

      \usepackage{amssymb}
			\usepackage{geometry}
			\usepackage{amsmath,amscd}
			\usepackage{hyperref}
			 
      \theoremstyle{plain}
      \newtheorem{theorem}{Theorem}[section]
      \newtheorem{lemma}[theorem]{Lemma}
      \newtheorem{coro}[theorem]{Corollary}

      \newtheorem{prop}{Proposition}[section]
      \newtheorem{defi}[theorem]{Definition}

      \theoremstyle{plain}
      \newtheorem{remark}[theorem]{Remark}

      \makeatletter
      \def\@setcopyright{}
      \def\serieslogo@{}
      \makeatother
   
\begin{document}

%



   \author{Adrien Boyer}
   \address{I2M, CMI, Aix-Marseille Univ. France}
   \email{aboyer@cmi.univ-mrs.fr}


  
   \title{Quasi-regular representations and rapid decay}


   \begin{abstract}
	We study \emph{property RD} in terms of rapid decay of matrix coefficients.
    We give new formulations of property RD in terms of a $L^{1}$-integra\-bility condition of a Banach representation. Combining this new definition with the existence of cyclic subgroups of exponential growth in non-uniform lattices in semisimple Lie groups, we deduce that there exist matrix coefficients associated to several kinds of quasi-regular representations which satisfy a ``non-RD condition'' for non-uniform lattices. We obtain also that such coefficients can not satisfy \emph{the weak inequality} of Harish-Chandra.      
   \end{abstract}

   \subjclass{Primary 43; Secondary 22}


  
   \dedicatory{}

   \date{\today}


   \maketitle



   \section{Introduction}

	We say that a unitary representation $\pi:G \rightarrow U(\mathcal{H})$ on a complex Hilbert space $\mathcal{H}$ of a locally compact group $G$ with a left invariant Haar measure $dg$ and with a length function $L$ has property RD with respect to to $L$ if there exist $d\geq 1$ and $C>0$, such that $\pi$ verifies for all $f\in L^{1}(G,dg)$, $$\|\pi(f)\| \leq C \|f\|_{L,d}$$ where $\|\cdot\|$ denotes the operator norm, and $$\|f\|_{L,d}=\left(\int_{G}|f(g)|^{2}(1+L(g))^{2d}dg\right)^{\frac{1}{2}}.$$ 
	
	Equivalently, a representation $\pi:G \rightarrow U(\mathcal{H})$ has property RD with respect to $L$ if there exist $d\geq 1$ and $C>0$ such that for each pair of unit vectors $\xi,\eta$ on $\mathcal{H}$ we have $$\int_{G}\left|\left\langle \pi(g)\xi,\eta\right\rangle\right|^{2}d\mu_{L,d}(g)\leq C$$ where $$d\mu_{d,L}(g)=\frac{dg}{(1+L(g))^{d}} \cdot $$  We say that a locally compact group has property RD with respect to $L$ if its regular representation has property RD with respect to $L$.

The property of rapid decay was introduced by U. Ha\-agerup at the end of the seventies in his work \cite{H}. Its essence could probably be traced back to Harish-Chandra's estimates of spherical functions on semisimple Lie groups and to the work of C. Herz \cite{He}. The terminology ``property RD" was introduced later, see the work \cite{Jo} of P. Jolissaint. Property RD is relevant in the context of the Baum-Connes conjecture thanks to the important work of V. Lafforgue in \cite{L}.

\subsection{Equivalent definitions of property RD, Hilbert-Schmidt and trace class operators}

Let $\pi:G \rightarrow U(\mathcal{H})$ be a unitary representation of $G$. For $\xi$ and $\eta$ in $\mathcal{H}$, we define the matrix coefficient associated to $\pi$, as $g\mapsto \left\langle \pi(g)\xi,\eta \right\rangle$. Let $\mathcal{L}^{2}(\mathcal{H})$ be the Hilbert space of Hilbert-Schmidt operators on the Hilbert space $\mathcal{H}$. Recall that the scalar product of Hilbert-Schmidt operators on $\mathcal{H}$ is defined as $\left\langle S,T\right\rangle =Tr(ST^{*})$ where $Tr$ denotes the usual trace on $B(\mathcal{H})$, and $T^{*}$ denotes the adjoint operator of $T$. Consider the representation $$c:G \rightarrow U(\mathcal{L}^{2}(\mathcal{H})),$$ defined by $$c(g)T=\pi(g)T\pi(g^{-1}), $$ for all $T$ in $\mathcal{L}^{2}(\mathcal{H})$. For $S$ and $T$ in $\mathcal{L}^{2}(\mathcal{H})$, the matrix coefficient associated to $c$ is $g\mapsto \left\langle c(g)S,T \right\rangle$. The Hilbert space $\mathcal{L}^{2}(\mathcal{H})$ contains the Banach space of trace class operators on the Hilbert space $\mathcal{H}$ denoted by $\mathcal{L}^{1}(\mathcal{H})$. The restriction of the representation $c$ to $\mathcal{L}^{1}(\mathcal{H})$ is a isometric Banach space representation for the norm $\|\cdot \|_{1}$ (see Subsection \ref{repre}).
In Section \ref{sec2} we prove the following proposition:
\begin{prop} \label{main}
Let G be a locally compact second countable group, and let $\mu$ be a Borel measure on $G$ which is finite on compact subsets. Let $\pi:G \rightarrow U(\mathcal{H})$ be a unitary representation on a Hilbert space, and let $c$ denote the corresponding Banach space representation defined above.
The following conditions are equivalent.

\begin{enumerate} 
\item There exists $C >0$ such that for all $S,T$  unit vectors in $\mathcal{L}^{1}(\mathcal{H})$ (unit for $\|\cdot\|_{1}$)  we have

 $\int_{G}\left|\left\langle c(g)S,T\right\rangle\right|d\mu(g)<C .$

\item There exists $C>0$, such that for all unit vectors $\xi,\eta$ in $\mathcal{H}$ we have 

$\int_{G}\left|\left\langle \pi(g)\xi,\eta \right\rangle\right|^{2}d\mu(g)<C .$

\item For all vectors $S,T$ in $\mathcal{L}^{1}(\mathcal{H})$ we have 

$\int_{G}\left|\left\langle c(g)S,T\right\rangle\right| d\mu(g)<\infty .$

\item For all vectors $\xi,\eta$ in $\mathcal{H}$ we have

$\int_{G}\left|\left\langle \pi(g)\xi,\eta \right\rangle\right|^{2}d\mu(g)<\infty .$
\end{enumerate}
\end{prop}

Let $\mu$ be the following measure  $$d\mu(g)=\frac{dg}{(1+L(g))^{d}},$$ where $dg$ denotes a Haar measure on $G$, and where $d$ is some positive real number. Applying Proposition \ref{main} with this special choice of measure provides us with four equivalent definitions of property RD for unitary representations. 
\\

\subsection{A simple condition for property RD for \emph{positive} representations}
Consider now a representation $\pi:G \rightarrow U(\mathcal{H})$ such that $\mathcal{H} \subset L^{2}(X,m)$ where $(X,m)$ is a measured space i.e. a vector $\xi$ in $\mathcal{H}$ is a complex valued function and for all $\xi$ and $\eta$ in $\mathcal{H}$ we have $\left\langle \xi,\eta \right\rangle=\int_{X}\xi(x)\overline{\eta(x)}dm(x)$. We say that a function $\xi$ is positive if $\xi \geq 0$ almost everywhere with respect to $m$. Let $$\mathcal{H}_{+}=\left\{\xi \in \mathcal{H} \right|   \xi \geq 0 ~~ \mbox{almost everywhere with respect to $m$} \} $$ be the cone of positive functions of $\mathcal{H}$. In the above definition $\mathcal{H}_{+}$ can be $\lbrace 0\rbrace$ but such a pathology never appears in this article. In this context, we say that $\pi$ is \textit{positive} if for all $g$ in $G$, we have $$\pi(g)\mathcal{H}_{+}\subset \mathcal{H}_{+}.$$ Typical examples of positive representations are provided by unitary representations coming from measurable actions $G\curvearrowright (X,m)$ and ``half densities" where $m$ is a $G$-quasi-invariant measure  (see Subsection \ref{sec3}).

 The following proposition proved in Section \ref{sec4} brings a simple condition for property RD for positive representations:

\begin{prop} \label{propositiv}
Let $G$ be a locally compact second countable group with a length function $L$. Let $\pi:G \rightarrow U(\mathcal{H})$ be a unitary representation such that $\mathcal{H}= L^{2}(X,m)$ with $(X,m)$ a measured space. Assume that $\pi$ is positive.
The following conditions are equivalent:

\begin{enumerate}
	\item $\pi$ has property RD with respect to $L$ i.e. there exist $d\geq 1$ and $C>0$ such that for all $\xi,\eta$ unit vectors in $\mathcal{H}$, we have $$\int_{G}\frac{\left|\left\langle \pi(g) \xi,\eta\right\rangle\right|^{2}}{(1+L(g))^{d}}dg<C.$$
	\item For each positive function $\xi \in \mathcal{H}_{+}$, there exists $d_{\xi}\geq 1$ such that $$\int_{G}\frac{ \left\langle \pi(g) \xi,\xi \right\rangle ^{2}}{(1+L(g))^{d_{\xi}}}dg<\infty.$$  
\end{enumerate}
\end{prop}

\subsection{Positive vectors}
We shall consider another notion of positivity. Let $\pi$ be a unitary representation of a locally compact group $G$ on a complex Hilbert space $\mathcal{H}$. Following Y. Shalom (see \cite[Theorem 2.2]{S}), we say that a nonzero vector $\xi \in \mathcal{H}$ is a \textit{positive vector} if it satisfies $$\left\langle \pi(g)\xi,\xi \right\rangle\geq 0, \forall g\in G.$$ This notion is particularly interesting for property RD. In fact, in order to prove property RD with respect to $L$ for a group (i.e. for the left regular representation) it  suffices to prove property RD with respect to $L$ for a representation with a positive vector.\\
Our goal is now to construct unitary representations without property RD.
 \subsection{Coefficients with slow decay and lattices}
Let $G$ be a locally compact group with a length function $L$ and let $\mu$ be a Haar measure on $G$ i.e. $\mu(B)=\int_{G}1_{B}dg$, where $1_{B}$ denotes the characteristic function of a Borel subset $B$ of $G$. Denote by $B_{L}(R)$ the ball (with respect to $L$) of radius $R$ whose center is the neutral element of $G$. We say that $G$ has polynomial growth with respect to $L$ if there exists a polynomial $P$ such that for all $R>0$ we have $$\mu(B_{L}(R))\leq P(R) .$$ It is easy to check that $G$ has polynomial growth with respect to $L$ if and only if there exists a positive number $d$ such that $$ \int_{G}\frac{dg}{(1+L(g))^{d}}<\infty . $$ Cauchy-Schwarz inequality implies that a unitary representation of a group of polynomial growth with respect to $L$ satisfies property RD with respect to $L$.

If a locally compact group $G$ admits a unitary representation satisfying property RD with respect to $L$ and with a non zero invariant vector, then $G$ must be a group of polynomial growth with respect to $L$. Therefore, we are interested in representations without non zero invariant vectors. 

Let $G$ be locally compact group. Consider an action $\alpha: G\curvearrowright (X,m)$  of $G$ on $(X,m)$ where $m$ is a $G$-quasi-invariant measure. Consider the unitary representation $$\pi_{\alpha}:G\rightarrow U\left(L^{2}(X,m)\right)$$ associated to this action on the Hilbert space $L^{2}(X,m)$. Observe that $\pi_{\alpha}$ is a positive representation (see Subsection \ref{sec3}).\\
Our goal is
now to construct representations without property RD.
\begin{theorem}\label{maintheo}
Let $\Gamma$ be a discrete countable group with a length function $L$. Consider an action  $\alpha:\Gamma \curvearrowright (X,m)$  on a $\sigma$-finite measured space $(X,m)$ with a $\Gamma$-quasi-invariant measure $m$. Consider $\pi_{\alpha}:\Gamma \rightarrow U(L^{2}(X,m))$ the unitary representation associated to $\alpha:\Gamma \curvearrowright (X,m)$. Assume that $\Gamma$ contains a cyclic subgroup of exponential growth with respect to $L$. Then there exists $\xi $ in $L^{2}(X,m)_{+}$ such that for all $d\geq 1$ we have $$\sum_{\gamma \in \Gamma}\frac{\left\langle \pi_{\alpha}(\gamma)\xi,\xi\right\rangle^{2}}{(1+L(\gamma))^{d}}=\infty.$$
\end{theorem}

The theorem is trivially true if $\pi_{\alpha}$ contains a non zero invariant vector which is in $L^{2}(X,m)_{+}$. For example consider the representations obtained from an action $\alpha:\Gamma \curvearrowright (X,m)$ where $m$ is a finite $\Gamma$-invariant measure. The constant function $1_{X}$ is a non zero invariant vector which is in $L^{2}(X,m)_{+}$. Examples of representations without non zero invariant vectors are described in Subsection \ref{invariant}. 
\\
In the following, Lie groups are endowed with a length function associated to a left-invariant Riemannian metric.
According to A. Lubotzky, S. Mozes and M.S Raghunatan (see \cite{Lu}), any non-uniform irreducible lattice in a higher rank semisimple Lie group contains a cyclic subgroup of exponential growth.

\begin{coro}\label{coro1}
Let $G$ be a connected non-compact simple Lie groups with finite center. Let $H$ be a closed subgroup of $G$. Let $\lambda_{G/H}:G \rightarrow U(L^{2}(G/H))$ be the corresponding quasi-regular representation. Let $\Gamma$ be a non-uniform lattice in $G$. Then there exists $\xi$ in $L^{2}(G/H)_{+}$ such that for all $d\geq 1$ we have $$\sum_{\gamma \in \Gamma}\frac{\left\langle \lambda_{G/H}(\gamma)\xi,\xi\right\rangle^{2}}{(1+L(\gamma))^{d}}=\infty.$$ 
\end{coro}

In Subsection \ref{qr'}, we assume that $G$ is locally compact second countable and can be written $G=KP$, where $K$ is a compact subgroup and $P$ is a closed subgroup which is not unimodular. Consider $\lambda_{G/P}:G \rightarrow U(L^{2}(G/P))$  the quasi-regular representation associated with $P$. 
The Harish-Chandra function 
\begin{equation}\label{defHch}
\Xi(g)=\left\langle \lambda_{G/P}(g)1_{G/P},1_{G/P}\right\rangle
\end{equation}
 is the diagonal coefficient of $ \lambda_{G/P}$ defined by the characteristic function  $1_{G/P}$ of the space $G/P$. Following Gangolli and Varadarajan, \cite[Definition 6.1.17]{GV}, we say that a function $f$ on the group $G$ equipped with a length function $L$, verifies the \emph{weak inequality} of Harish-Chandra if there exist $C>0$ and $d\geq 0$ such that 
\begin{equation} \label{inegfaible}
|f(g)|\leq C(1+L(g))^{d}\Xi(g).
\end{equation}
We prove:
\begin{coro}\label{coro2}
Let $G$ be a non-compact semisimple Lie group with finite center. Let $H$ be a closed subgroup of $G$. Let $\lambda_{G/H}:G\rightarrow U(L^{2}(G/H))$ be the corresponding quasi-regular representation. Then there exist $\xi$ and $\eta$ in $L^{2}(G/H)$ such that $g\mapsto \left\langle \lambda_{G/H}(g)\xi,\eta \right\rangle$ does not satisfy the \emph{weak inequality} of Harish-Chandra. 
\end{coro}

\begin{remark}

Corollary \ref{coro2} holds true for $H:=P$ a minimal parabolic subgroup of a non-compact semisimple Lie group with finite center. Although it is known that a coefficient $g\mapsto \left\langle \lambda_{G/P}(g)1_{G/P},\eta \right\rangle$ with $\eta \in L^{2}(G/P)$ not in $L^{\infty}(G/P)$, does not satisfy the weak inequality (it is a consequence of Fatou's Theorem for semisimple Lie groups, see \cite[Theorem 5.1]{Sj}), Corollary \ref{coro2} gives a new proof of this result. More generally, Corollary \ref{coro2} implies that there exist matrix coefficients associated to any quasi-regular representations that do not satisfy the weak inequality of Harish-Chandra. Although the matrix coefficient $g\mapsto \left\langle \lambda_{G/P}(g)1_{G/P},\eta \right\rangle$ with $\eta \in L^{2}(G/P)$ not in $L^{\infty}(G/P)$ does not satisfy the weak inequality, we prove in Section 4 the following:
\begin{prop}\label{uncoef}
Let $G$ be a non-compact semisimple Lie group with finite center. Let $\Gamma$ be any discrete subgroup of $G$. Consider $P$ a minimal parabolic subgroup of $G$, and $\lambda_{G/P}:G \rightarrow U(L^{2}(G/P))$ the quasi-regular representation associated to $P$. Then there exist $C>0$ and $d\geq 1 $ such that for all $\eta \in L^{2}(G/P)$ we have $$\sum_{\gamma \in \Gamma}\frac{\left|\left\langle \lambda_{G/P}(\gamma)1_{G/P},\eta\right\rangle \right|^{2}}{(1+L(\gamma))^{d}} \leq C\| \eta \|^{2}$$ where $\| \cdot \|$ denotes the $L^{2}$-norm on $L^{2}(G/P)$. 
\end{prop}

\end{remark}

\subsection{Acknowledgements}
I would like to thank C. Pittet for very helpful discussions and criticisms. I would like also to thank the referees for their useful remarks and comments.


   \section{From square integrable representations to integrable representations}\label{sec2}

\subsection{Representations}\label{repre}
Let $G$ be a locally compact group. Let $\pi: G \rightarrow U(\mathcal{H})$ be a unitary representation on a Hilbert space. Consider $c$ as in the introduction.
Write $T=\sum_{i}\alpha_{i} \left\langle ~,\xi_{i}\right\rangle \eta_{i}$. We have $$c(g)T=\sum_{i}\alpha_{i} \left\langle ~,\pi(g)\xi_{i}\right\rangle \pi(g)\eta_{i}.$$ Observe that $\|c(g)T\|_{2}=\|\alpha\|_{l^{2}}=\|T\|_{2}$ where  $\|T\|_{2}$ denotes the norm of a Hilbert-Schmidt operator $T$ and $\|c(g)T\|_{1}=\|\alpha\|_{l^{1}}=\|T\|_{1}$ where $\|T\|_{1}$ denotes the norm of a trace class operator $T$. Hence, $c$ is a unitary representation on $\mathcal{L}^{2}(\mathcal{H})$ and it is an isometric Banach space representation on $\mathcal{L}^{1}(\mathcal{H})$.
Let $\overline{\pi}$ be the conjugate unitary representation on $\overline{\mathcal{H}}$ of $\pi$.
 Let $\sigma$ be the unitary representation: 
\begin{eqnarray*}
\sigma:G  &\rightarrow & U(\overline{\mathcal{H}} \otimes \mathcal{H}) \\
g &\mapsto& \overline{\pi}(g) \otimes \pi(g) .
\end{eqnarray*}

The Banach space isomorphism $\Phi$ defined as
\begin{eqnarray*}
\Phi: \overline{\mathcal{H}} \otimes \mathcal{H}  & \rightarrow & \mathcal{L}^{2}(\mathcal{H}) \\
\xi \otimes \eta &\mapsto& \left\langle ~,\xi \right\rangle \eta
\end{eqnarray*}
 intertwines $\sigma$ and the representation $c$: $$c\Phi =\Phi \sigma,$$ and this equivalence restricts to an equivalence between Banach space representations. For more details see  \cite[Chap. 2, \S~2.1, p.~12]{R}, \cite[Chap. 9.1, from \S~9.1.31 to 9.1.38]{D} and \cite[Chap. 1, \S~6, p. 96]{Pa} .

\begin{lemma}\label{commute}
Let $U,V$ be a pair of unit vectors in $\overline{\mathcal{H}} \widehat{\otimes} \mathcal{H} $ where $\widehat{\otimes}$ denotes the projective tensor product of Banach spaces. There exists a unique pair of unit vectors $S,T $ in $\mathcal{L}^{1}(\mathcal{H})$ such that $$\left\langle \sigma(g)U,V\right\rangle=\left\langle c(g)S,T\right\rangle.$$
\end{lemma}
\begin{proof}
We can write $U=\Phi(S)$ and $V=\Phi(T)$ for a unique pair $S,T \in \mathcal{L}^{1}(\mathcal{H})$ of unit vectors because $\Phi$ is an isomorphim of Banach spaces. Furthermore, because $\Phi$ is an isomorphism of Hilbert spaces $\overline{\mathcal{H}}\otimes \mathcal{H} \supset \overline{\mathcal{H}}\widehat{\otimes}\mathcal{H}$ and $\mathcal{L}^{2}(\mathcal{H})\supset  \mathcal{L}^{1}(\mathcal{H})$ we have 
\begin{eqnarray*}
\left\langle \sigma(g)U,V\right\rangle&=& \left\langle \sigma(g) \Phi(S),\Phi(T) \right\rangle \\
                                      &=& \left\langle \Phi^{-1}\sigma(g) \Phi(S),T\right\rangle\\
						                          &=&  \left\langle c(g) S,T \right\rangle .   
																			\end{eqnarray*}

\end{proof}

\subsection{Proof of Proposition \ref{main}}
\subsubsection{An application of the Banach-Steinhaus theorem}
The next proposition is an application of the Banach-Steinhaus theorem. It enables us to prove Implication $(4) \Rightarrow (1)$ from Proposition \ref{main} in the next subsection.
\begin{prop}\label{multi}
Let $B: X_{1}\times X_{2} \times ... \times X_{r} \rightarrow \mathbb{C}$ be a multilinear map on a product of Banach spaces. If $B$ is continuous on each variable, then $B$ is continuous.
\end{prop}
\begin{proof}
By induction on $r$. 
See \cite[p.~81, Corollary]{RS} for the case $r=2$.
\end{proof}
If $G$ is a locally compact second countable group, we denote by $d\mu(g)$ a Borel measure on $G$ which is finite on compact subsets of $G$.

\begin{prop}\label{enfin}
Let $G$ be a locally compact second countable group. Let $\pi:G \rightarrow U(\mathcal{H})$ be a unitary representation on a Hilbert space and let $\sigma=\overline{\pi} \otimes \pi$. If for all $\xi,\eta \in \mathcal{H}$ we have $\int_{G}\left|\left\langle \pi(g)\xi,\eta \right\rangle\right|^{2}d\mu(g)<\infty$, then there exists $C>0$ such that for all unit vectors $\xi,\xi',\eta,\eta' \in \mathcal{H}$ we have $\int_{G}\left|\left\langle \sigma(g)\xi \otimes \eta, \xi'\otimes \eta' \right\rangle\right|d\mu(g)\leq C .$

\end{prop}

\begin{proof}
Let $G=\cup K_{n}$, where $K_{n}$ is an exhausting sequence of compact subsets of $G$.
Fix $\eta,\xi',\eta' \in \mathcal{H}$, and let us define the family $T_{n}$ of continuous linear operators:
\begin{align*}
T_{n}:\mathcal{H} &\rightarrow L^{1}(G,\mu)\\
\xi &\longmapsto \big( g \mapsto 1_{K_{n}}(g)\left\langle \sigma(g)  \xi \otimes \eta, \xi' \otimes \eta' \right\rangle \big) 
\end{align*}
 with $1_{K_{n}}$ the characteristic function of $K_{n}$. 
	
We have for each $\xi \in \mathcal{H}$:
\begin{align*}
\sup_{n} \|T_{n}(\xi)\|_{L^{1}}&=\sup_{n}\int_{K_{n}}\left|\left\langle \sigma(g)  \xi \otimes \eta, \xi' \otimes \eta' \right\rangle\right| d\mu(g) \\
&=\sup_{n}\int_{K_{n}}\left|\left\langle \pi(g)  \xi, \xi' \right\rangle\right|\left|\left\langle \pi(g)  \eta, \eta' \right\rangle\right| d\mu(g)  \\
&\leq \sup_{n} \left\{ \left(\int_{K_{n}}\left|\left\langle \pi(g)  \xi, \xi' \right\rangle\right|^{2} d\mu(g) \right)^{\frac{1}{2}} \left(\int_{K_{n}}\left|\left\langle \pi(g)  \eta, \eta' \right\rangle\right|^{2} d\mu(g) \right)^{\frac{1}{2}} \right\}\\
& \leq \left(\int_{G}\left|\left\langle \pi(g)  \xi, \xi' \right\rangle\right|^{2} d\mu(g) \right)^{\frac{1}{2}} \left(\int_{G}\left|\left\langle \pi(g)  \eta, \eta' \right\rangle\right|^{2} d\mu(g) \right)^{\frac{1}{2}}\\
 &<\infty,
\end{align*}
by hypothesis. The Banach-Steinhaus theorem implies that $\sup_{n}\|T_{n}\|<\infty$ where $\|\cdot\|$ denotes the operator norm. Hence   
$$\sup_{ \left\{\xi,\|\xi\|=1 \right\}} \int_{G}\left|\left\langle \sigma(g)  \xi \otimes \eta, \xi' \otimes \eta' \right\rangle\right| d\mu(g)<\infty.$$
	
	This proves that the multilinear form
	$B:$
	\begin{align*}
	\overline{\mathcal{H}}\times \mathcal{H} \times \mathcal{H} \times \overline{\mathcal{H}} &\rightarrow \mathbb{C} \\
(\xi,\eta,\xi',\eta') &\mapsto \int_{G}\left\langle \sigma(g)  \xi \otimes \eta, \xi' \otimes \eta' \right\rangle d\mu(g)
\end{align*}
is continuous in $\xi$. Analogous arguments show that $B$ is continuous on $\eta,\xi'$ and $\eta'$. Proposition \ref{multi} completes the proof.

\end{proof}

\subsubsection{Proof of Proposition \ref{main}}
\begin{proof}
$(1) \Rightarrow (3)$ and $(2) \Rightarrow (4)$  are obvious.
\\
$(1)\Rightarrow (2)$ and $(3) \Rightarrow (4)$ are clear:
take $U=\xi \otimes \xi$ and $V=\eta \otimes \eta$. Then 
$$\left\langle \sigma(g)U,V\right\rangle= \left|\left\langle \pi(g)\xi ,\eta \right \rangle\right|^{2}.$$
Lemma $\ref{commute}$ and integration complete the proof.\\
Let us prove
$(4) \Rightarrow (1):$ 

Take $U,V$ two unit vectors (for the projective norm) in $\overline{\mathcal{H}}\widehat{\otimes}\mathcal{H}$. Write $U=\sum_{k}\alpha_{k} \xi_{k} \otimes \eta_{k}$ and $V=\sum_{l}\beta_{l} \xi'_{l} \otimes \eta'_{l}$ with $\|\alpha\|_{l^{1}}=1=\|\beta\|_{l^{1}}$, where $(\xi_{k})_{k\in\mathbb{N}},(\xi'_{k})_{k\in\mathbb{N}}$ and  $(\eta_{k})_{k\in\mathbb{N}},(\eta'_{k})_{k\in\mathbb{N}}$ are orthonormal families in $\overline{\mathcal{H}}$ and $\mathcal{H}$. We have :

\begin{align*}
\int_{G}\left|\left\langle \sigma(g)U,V\right\rangle\right|d\mu(g)&= \int_{G}\left|\left\langle \sigma(g) \sum_{k}\alpha_{k} \xi_{k} \otimes \eta_{k},\sum_{l}\beta_{l} \xi'_{l} \otimes \eta'_{l}  \right\rangle\right|d\mu(g) \\
& \leq \sum_{k,l}\left|\alpha_{k}\right|\left|\beta_{l}\right|\int_{G}\left|\left\langle \sigma(g)\xi_{k}\otimes \eta_{k} ,\xi'_{l}\otimes \eta'_{l}\right\rangle\right|d\mu(g).
\end{align*}

Thanks to Proposition \ref{enfin}, there exists $C>0$ such that for all $k,l\in\mathbb{N}$ $$\int_{G}\left|\left\langle \sigma(g)\xi_{k}\otimes \eta_{k} ,\xi'_{l}\otimes \eta'_{l}\right\rangle\right|d\mu(g)\leq C.$$
Hence,
\begin{align*}
\int_{G}\left|\left\langle \sigma(g)U,V\right\rangle\right|d\mu(g) & \leq \sum_{k,l}\left|\alpha_{k}\right|\left|\beta_{l}\right|C\\
&=C.
\end{align*}
The proof of $(2) \Rightarrow (3)$ is similar and left to the reader.
 \end{proof}

\section{Length functions, quasi-regular representations}
		\subsection{Property RD}
	We shall define property RD for a unitary representation of a locally compact group $G$.
	\begin{defi}\label{def1}
	A length function $L:G \rightarrow \mathbb{R_{+}}$ on a locally compact group is a measurable function which is locally bounded ( i.e. for any compact $K\subset G$, we have $\sup\left\{L(g),g\in K\right\}<\infty$) satisfying

	\begin{enumerate}
	\item $L(e)=0$,
    \item $L(g^{-1})=L(g)$,
	\item  $L(gh)\leq L(g)+L(h).$
	\end{enumerate}
	
	\end{defi}
	\begin{remark}
	Assume that $G\curvearrowright(X,d)$ acts properly by isometries on a metric space. Fix a point $x\in X$, then $g\mapsto L(g):=d(g\cdot x,x)$ defines a length function on $G$. 
	\end{remark}
	
	We will need the following:
	\begin{lemma}\label{fonctionlong}
	Let $G$ be a locally compact group endowed with a length function $L$. Let $O$ be a compact subset in $G$. Then there exists a positive constant $C$ such that $$1+L(gx)\geq C(1+L(g))$$ for all $g$ in $G$ and for all $x$ in $O$.
	\end{lemma}
	\begin{proof}
	Triangle inequality and the equality $L(x)=L(x^{-1})$ imply that for all $g,x\in G$ $$1+L(g)\leq \left(1+L(gx)\right)\left(1+L(x)\right).$$ Since $L$ is locally bounded, we have for all $x$ in $O$ 
	$$(1+L(g))\leq \left(1+L(gx)\right)\left(1+M\right),$$ where $M=\sup\left\{L(x),x\in O\right\}.$ 
	\end{proof}
	
		\begin{defi}
	A unitary representation $\pi: G \rightarrow U(\mathcal{H})$ of a locally compact group has property RD with respect to $L$ if there exist $C>0$ and $d\geq 1$ such that for each pair of unit vectors $\xi,\eta \in \mathcal{H}$ we have $$\int_{G}\frac{\left|\left\langle \pi(g)\xi,\eta\right\rangle\right|^{2}}{(1+L(g))^{d}}dg\leq C.$$ 
	\end{defi}

	\begin{remark}\label{rm1}
	Assume $K$ is a symmetric relatively compact generating set of $G$
containing the identity. Then $L(g) = \min \left\{ n\in  \mathbb{N}: \exists k_{1},...,k_{n}~~ such~~that~~ g=k_{1}...k_{n} \right\}$ is a length function on G. If $L'$ is any length function
on $G$, then $L'(g) \leq C L(g),$ where $C= \sup\left\{ L(k),k\in K \right\}$. Hence if $\pi$ has rapid decay with respect to $L'$ then
it has rapid decay with respect to L.
\end{remark}
\subsection{Representations associated to measurable actions}\label{sec3}
Let $G$ be a locally compact group acting measurably on a measured space $ \alpha: G\curvearrowright (X,m)$ where $m$ is a $G$-quasi-invariant measure. The action is measurable in the sense that: the map $$\alpha:(g,x) \in G\times X \mapsto \alpha(g,x)=g\cdot x \in X ,$$ is measurable. For all $g\in G$ we denote by $g_{*}m$ the measure which verifies for all Borel subsets $A\subset X$ $$g_{*}m(A)=m(g^{-1}A).$$ We say that $m$ is $G$-quasi-invariant if for all $g\in G$, $m$ and $g_{*}m$ are in the same measure class. We say that $m$ is $G$-invariant if for all $g\in G$, $g_{*}m=m$. We denote by $$x\in X \mapsto \frac{dg_{*}m}{dm}(x)$$ the Radon-Nikodym derivative of the measure $g_{*}m$ with respect to $m$. It verifies $$\int_{X}f(g\cdot x)\frac{dg^{-1}_{*}m}{dm}(x) dm(x)=\int_{X}f(x)dm(x).$$ In this situation, consider the Hilbert space $L^{2}(X,m)$, and define the unitary representation $\pi_{\alpha}:G\rightarrow U(L^{2}(X,m))$ as 
\begin{equation}\label{defiqr}
\pi_{\alpha}(g)\xi(x)=\left(\frac{dg_{*}m}{dm}(x)\right)^{\frac{1}{2}}\xi(g^{-1}\cdot x).
\end{equation}

In other words, $$ \left(\frac{dg_{*}m}{dm}(x)\right)^{\frac{1}{2}}dm(x)$$ can be seen as a ``half density'' on $X$.
\subsection{Quasi-regular representations}\label{qr}
Quasi-regular representations associated to a pair $(G,H)$ where $H$ is a closed subgroup of $G$ provide examples of unitary representations associated to a measurable action. But first of all, we recall what is the measure class we consider on $G/H.$ 
\subsubsection{A measure class on $G/H$}

Let $G$ be a locally compact group with a Haar measure $dg$, and let $H$ be a closed subgroup of $G$ with a Haar measure $dh$. The space $G/H$ is endowed with its quotient topology. Consider its Borel $\sigma$-algebra. We shall define a Borel measure on $G/H$. 
Let $\Delta_{G}$ and $\Delta_{H}$ be the modular functions of $G$ and $H$. A \textit{rho-function} for the pair $(G,H)$ is a continuous map $\rho:G \rightarrow \mathbb{R}^{*}_{+}$ satisfying for all $g\in G$ and for all $h\in H$ $$ \rho(gh)=\frac{\Delta_{H}(h)}{\Delta_{G}(h)}\rho(g).$$
It always exists, see \cite[(2.54)]{Fo} or \cite[Chapter 8, Section 1]{Re}. Therefore, given a rho-function for the pair $(G,H)$, there exists a quasi-invariant regular Borel measure on $G/H$ denoted by $\nu$ such that for all $f\in C_{c}(G),$ $$\int_{G}f(g)\rho(g)dg=\int_{G/H}\int_{H}f(gh)dhd\nu(gH),$$ with  Radon-Nikodym derivative $$\frac{dg^{-1}_{*}\nu}{d\nu}(xH)=\frac{\rho(gx)}{\rho(x)},$$ for all $g,x\in G$. See \cite[Appendix B]{Be}. 
The quasi-regular representation associated to a pair to $(G,H)$ is the unitary representation 
$\lambda_{G/H}:G \rightarrow U(L^{2}(G/H,\nu))$ defined as 
 $$\left(\lambda_{G/H}(g)\xi\right)(xH)=\left(\frac{\rho(g^{-1}x)}{\rho(x)}\right)^{\frac{1}{2}}\xi(g^{-1}xH)$$

 for all $\xi\in L^{2}(G/H,\nu)$, for all $g\in G$ and for all $xH \in G/H$.

\subsubsection{A particular class of quasi-regular representations}\label{qr'}
	
Let $G$ be a locally compact group which is unimodular, and assume that there exists a compact subgroup $K$ and a closed subgroup $P$ of $G$ such that $$G = KP.$$We shall define a rho-function for the pair $(G,K)$. We denote by $\Delta_{P}$ the right-modular function of $P$. We extend to $G$ the map $\Delta_{P}$ of $P$ as $\Delta : G \rightarrow \mathbb{R}^{*}_{+}$ with $\Delta(g)=\Delta(kp):=\Delta_{P}(p)$. It is well defined because $K\cap P$ is compact (observe that $\Delta_{P}|_{K\cap P}=1$). Notice that for all $g\in G$ and for all $p\in P$, $\Delta(gp)=\Delta(g)\Delta(p)=\Delta(g)\Delta_{P}(p)$. Hence the function
 $$x \in G \mapsto \Delta(x) \in \mathbb{R^{*}_{+}},$$ defines a $\rho$ function. Observe also that
 $$x \in G/P \mapsto  \frac{\Delta(gx)}{\Delta(x)} \in \mathbb{R^{*}_{+}}$$ is well defined. The quotient $G/P$ carries a unique $G$-quasi-invariant measure $\nu$, such that the Radon-Nikodym derivative at $(g,x)\in G \times G/P$ denoted by  $\kappa(g,x)=\frac{dg_{*}\nu}{d\nu}(x)$  satisfies $$\frac{dg^{-1}_{*}\nu}{d\nu}(x)=\frac{\Delta(gx)}{\Delta(x)}$$ for all $g\in G$ and $x\in G/P$.  Consider the quasi-regular representation  $\lambda_{G/P}:G \rightarrow U(L^{2}(G/P))$ associated to $P$, defined as $$(\lambda_{G/P}(g)\xi)(x)=\kappa(g,x)^{\frac{1}{2}}\xi(g^{-1}\cdot x).$$ We denote by $1_{G/P}$ the characteristic function of $G/P$. 

\begin{defi}
The \emph{Harish-Chandra function} $\Xi :G \rightarrow (0,\infty)$ is defined as 
$$\Xi(g):=\left\langle\lambda_{G /P}(g)1_{G/P},1_{G/P} \right\rangle.$$
\end{defi} 
As $\lambda_{G/P}$ is a unitary representation, we have $\Xi(g)=\Xi(g^{-1})$ for all $g\in G$. Observe also that for all $k,k'\in K$ we have $\Xi(kgk')=\Xi(g).$

\subsection{Stability of some matrix coefficients}

 Let $\Gamma$ be a discrete subgroup of a locally compact group $G$. Let $\pi: G \rightarrow U(\mathcal{H})$ be a unitary representation. A matrix coefficient, $g\mapsto \left\langle \pi(g) \xi,\eta \right\rangle$ is \emph{stable} on $G$ relative to $\Gamma$ if there exists a relatively compact neighborhood $V$ of the neutral element $e\in G$ and $C,c>0$ such that $$c\left|\left\langle \pi(\gamma) \xi,\eta \right\rangle\right|\leq  \left|\left\langle \pi(\gamma g) \xi,\eta \right\rangle\right| \leq C\left|\left\langle \pi(\gamma) \xi,\eta \right\rangle\right|$$ for all $g\in V$ and for all $\gamma \in \Gamma$. 

The interest of stable matrix coefficients is illustrated by the following proposition:
\begin{lemma}\label{l1}
Let $\Gamma$ be a discrete subgroup of a locally compact group $G$. Let $\pi: G \rightarrow U(\mathcal{H})$ be a unitary representation, and let $L$ be a length function on $G$. Let $\xi$ and $\eta$ be in $\mathcal{H}$. If $g\mapsto \left\langle \pi(g)\xi,\eta \right\rangle$ is a stable coefficient relative to $\Gamma$, then there exists a constant $C\geq 1 $ such that for all $d\geq 1$ we have:
$$\sum_{\Gamma} \frac{\left|\left\langle \pi(\gamma)\xi,\eta \right\rangle\right|^{2}}{(1+L(\gamma))^{d}}\leq C \int_{G}\frac{\left|\left\langle \pi(g)\xi,\eta \right\rangle\right|^{2}}{(1+L(g))^{d}}dg.$$

\end{lemma}
\begin{proof}
Let $V$ be a relatively compact neighborhood of the neutral element of $G$, such that $\gamma \cdot V \cap \gamma' \cdot V=\emptyset $ for all $\gamma,\gamma' \in \Gamma$ such that $\gamma\neq\gamma'$.
Consider a matrix coefficient $g\mapsto\left\langle \pi(g)\xi,\eta \right\rangle$ which is stable relative to $\Gamma$. We have
\begin{eqnarray*}
\int_{G}\frac{\left|\left\langle \pi(g)\xi,\eta \right\rangle\right|^{2}}{(1+L(g))^{d}}dg &\geq &\sum_{\gamma} \int_{\gamma.V}\frac{\left|\left\langle \pi(g)\xi,\eta \right\rangle\right|^{2}}{(1+L(g))^{d}}dg      \\
                                                                                                &=&\sum_{\gamma} \int_{V}\frac{\left|\left\langle \pi(\gamma x )\xi,\eta \right\rangle\right|^{2}}{(1+L(\gamma x))^{d}}dx     \\ 
																																																&\geq& \sum_{\gamma} \int_{V}c^{2}\frac{\left|\left\langle \pi(\gamma)\xi,\eta \right\rangle\right|^{2}}{(1+L(\gamma x))^{d}}dx \\
																																																&\geq& c' \sum_{\gamma} \frac{\left|\left\langle \pi(\gamma)\xi,\eta \right\rangle\right|^{2}}{(1+L(\gamma))^{d}},
\end{eqnarray*} 
for some positive constant $c'$ depending on $V$ and on the constant of Lemma \ref{fonctionlong}.
\end{proof}

\begin{lemma}\label{stable'}
Let $\eta$ be in $L_{+}^{2}(G/P)$.
The matrix coefficient $g\mapsto \left\langle \lambda_{G/P}(g)1_{G/P},\eta\right\rangle$ is stable relative to every discrete subgroup of $G$.
\end{lemma}

\begin{proof}
Let $\Gamma$ be a discrete subgroup of $G$. Let $V$ be a relatively compact neighborhood of $e$ in $G$, sufficiently small so that $\gamma \cdot V \cap \gamma'\cdot V =\emptyset$ for all $\gamma \neq \gamma' \in \Gamma$. We have $\left\langle \lambda_{G/P}(\gamma g)1_{G/P},\eta \right\rangle=\left\langle \lambda_{G/P}(g) 1_{G/P},\lambda_{G/P}(\gamma^{-1})\eta \right\rangle.$ That is $$\left\langle \lambda_{G/P}(\gamma g)1_{G/P},\eta \right\rangle=\int_{G/P}\kappa(g, x)^{\frac{1}{2}}\kappa(\gamma^{-1} , x)^{\frac{1}{2}}\eta(\gamma \cdot x) d\nu(x).$$ The function $(g,x)\in G\times G/P \mapsto \kappa(g,x)^{\frac{1}{2}}$ is a strictly positive continuous function. Therefore, as $\overline{V}$ and $G/P$ are compact, there exist $C,c>0$ such that for all $g\in \overline{V}$ and for all $x\in G/P$, we have $$c\leq\kappa(g , x)^{\frac{1}{2}}\leq C.$$ Notice that $\lambda_{G/P}$ is a positive representation. Therefore, since $\eta$ is in $L_{+}^{2}(G/P)$, we obtain for all $\gamma \in \Gamma$, for all $g\in V$ $$ c \left\langle \lambda_{G/P}(\gamma)1,\eta\right\rangle\leq \left\langle \lambda_{G/P}(\gamma g)1_{G/P},\eta \right\rangle\leq C \left\langle \lambda_{G/P}(\gamma )1_{G/P},\eta\right\rangle  .$$ 
\end{proof}

We obtain immediately that the Harish-Chandra function is stable relative to every discrete subgroup of $G$:
\begin{coro}$($\cite[\mbox{Proposition 4.6.3, p.~159}]{GV}$.)$ \label{stable}
The Harish-Chandra function is stable relative to every discrete subgroup of G.
\end{coro}

Combining Corollary \ref{stable} with Lemma \ref{l1}, we obtain the following:

\begin{prop}\label{hcdiscret}
Let $G$ be a locally compact group decomposed as $G=KP$ where $K$ is a compact subgroup and $P$ is a closed subgroup. Let $\Gamma$ be a discrete subgroup of $G$ and let $L$ be a length function on $G$. There exists a constant $C>0$ such that for all $d\geq 1$ we have $$\sum_{\gamma \in \Gamma}\frac{\Xi^{2}(\gamma)}{(1+L(\gamma))^{d}}\leq C \int_{G}\frac{\Xi^{2}(g)}{(1+L(g))^{d}}dg.$$ 
\end{prop}

The representation $\sigma:G \rightarrow U(\overline{\mathcal{H}} \otimes \mathcal{H})$ introduced in Subsection \ref{repre} satisfies for all $\xi,\eta \in \mathcal{H}$: $$\left\langle \sigma(g) \xi \otimes \xi, \eta \otimes \eta \right\rangle=\left|\left\langle\pi(g)\xi,\eta \right\rangle\right|^{2}.$$

The representation $\sigma$ can be used to give a short and elementary proof of the following result, due to C. Herz.

\begin{theorem} $($\cite{He}, C. Herz.$)$ \label{Rd}
 Let G be a connected semisimple Lie group with finite center. Then $G$ has property RD with respect to a length function associated to a left-invariant Riemannian metric on $G$.
\end{theorem}
See \cite{He},\cite{C},\cite{Boy} for proofs.

We can now easily prove Proposition \ref{uncoef}:

\begin{proof}
Consider the quasi-regular representation $\lambda_{G/P}$ associated to $P$ a minimal parabolic subgroup of $G$. In \cite{Boy}, we prove that this representation satifies property RD with respect to  $L$ where $L$ is associated to a left-invariant Riemannian metric. Hence, there exist $C>0$ and $d\geq1$ such that for all $\xi,\eta \in L^{2}(G/P)$ we have $$\int_{G}\frac{|\left\langle \pi(g)\xi,\eta \right\rangle|^{2}}{(1+L(g))^{d}}dg \leq C\| \xi\|^{2} \| \eta\|^{2}.$$
Applying the above inequality for $\xi=1_{G/P}$ and $\eta\in L_{+}^{2}(G/P)$,
and using Lemma \ref{stable'} and Lemma \ref{l1} we obtain for some $C'>0$ and for the same $d\geq 1$, that for all $\eta \in L^{2}_{+}(G/P)$
$$\sum_{\gamma \in \Gamma}\frac{\left|\left\langle \lambda_{G/P}(\gamma)1_{G/P},\eta\right\rangle \right|^{2}}{(1+L(\gamma))^{d}} \leq C'\| \eta \|.$$  
\end{proof}

\section{Proofs}\label{sec4}
\subsection{Proof of Proposition \ref{propositiv}}

We start by a lemma about positive representations.
	\begin{lemma}\label{vecteurspositifs}
	
Let $G$ be  a locally compact second countable group with a length function $L$. Let $\pi:G \rightarrow U(\mathcal{H})$ be a unitary representation with $\mathcal{H}=L^{2}(X,m)$ where $(X,m)$ is a measured space. Assume that $\pi$ is positive ($\pi $ preserves the cone of positive functions). The following assertions are equivalent.
\begin{enumerate}
	\item There exists $d\geq1$ such that for all vectors $\xi,\eta \in \mathcal{H}$ we have $$\int_{G}\frac{\left|\left\langle \pi(g) \xi,\eta \right\rangle\right|^{2}}{(1+L(g))^{d}}dg<\infty .$$ 
	\item There exists $d\geq1$ such that for all $\xi \in \mathcal{H}_{+}$ we have $$\int_{G}\frac{\left\langle \pi(g) \xi,\xi \right\rangle ^{2}}{(1+L(g))^{d}}dg<\infty.$$	
\end{enumerate}
\end{lemma}

\begin{proof}
$(1) \Rightarrow (2)$ is obvious.

Let us prove $(2) \Rightarrow (1)$.  Observe first that the decomposition of a real valued function $\xi$ into its positive and negative part satisfies $|\xi_{+}-\xi_{-}|\leq \xi_{+}+\xi_{-}$. By positivity of $\pi$, for all $g\in G$ we have $\pi(g)|\xi_{+}-\xi_{-}|\leq \pi(g)\xi_{+}+ \pi(g)\xi_{-}.$ 
  Using the decomposition of a complex valued function into its real and imaginary parts, and the decomposition of a real valued function into its positive and negative parts, we obtain, according to the above observation, that $|\left\langle \pi(g)\xi,\eta\right\rangle|$ is less than or equal to a linear combination of matrix coefficients $\left\langle\pi(g)\xi',\eta' \right\rangle$ with  $\xi',\eta'$ positive vectors in $\mathcal{H}$. Now observe that for positive vectors, $$\left\langle \pi(g)\xi',\eta'\right\rangle \leq \left\langle \pi(g)(\xi'+\eta'),(\xi'+\eta')\right\rangle.$$
	Integration and Cauchy-Schwarz inequality complete the proof.
	
\end{proof}

We prove Proposition \ref{propositiv}:
\begin{proof}
$(1) \Rightarrow (2)$ is clear.\\
$(2) \Rightarrow (1)$. Thanks to the equivalence between $(2)$ and $(4)$ in Proposition \ref{main} with $d\mu(g)=\frac{dg}{(1+L(g))^{d}}$ for some $d\geq 1$, it is enough to prove that there exists $d \geq 1$, such that for each pair of vectors $\xi,\eta \in \mathcal{H}$ we have $$\int_{G}\frac{\left|\left\langle \pi(g) \xi,\eta\right\rangle\right|^{2}}{(1+L(g))^{d}}dg<\infty.$$ According to Lemma $\ref{vecteurspositifs}$ it is equivalent to prove that there exists $d\geq 1$ such that for any positive function $\xi \in L^{2}(X,m)$ we have $$\int_{G}\frac{ \left\langle \pi(g) \xi,\xi\right\rangle^{2}}{ (1+L(g))^{d} } dg<\infty.$$ We give a proof by contraposition. 

Assume that for each $n \in \mathbb{N}$, there exists a unit positive vector $\xi_{n}$ such that $$\int_{G}\frac{\left\langle \pi(g) \xi_{n},\xi_{n}\right\rangle ^{2}}{(1+L(g))^{n}} dg=\infty.$$ Take a sequence $(a_{n})_{n\in\mathbb{N}}$ of strictly positive real numbers such that the series $\sum_{n}a_{n}$ converges. We consider $$\xi=\sum_{n}a_{n}\xi_{n}$$ which is a well defined positive element of $\mathcal{H}$. We can assume that $\xi\neq 0$. We can replace $\xi$ by $\frac{\xi}{\|\xi\|}$ which is a unit vector, so we assume that $\xi$ is a unit vector. Let $d$ be a positive real number. Let $n$ be an integer such that $n\geq d$. Notice that $\left\langle \pi(g) \xi_{n},\xi_{m}\right\rangle \geq 0$ for all $n,m \in \mathbb{N}$ and for all $g$ in $G$. Hence,
\begin{eqnarray*}
\int_{G}\frac{\left\langle \pi(g) \xi,\xi\right\rangle ^{2}}{(1+L(g))^{d}}dg&\geq& \int_{G}\frac{\left\langle \pi(g) \xi,\xi\right\rangle ^{2}}{(1+L(g))^{n}}dg \\
                                                                         &\geq& a_{n}^{4}\int_{G}\frac{\left\langle \pi(g) \xi_{n},\xi_{n}\right\rangle ^{2}}{(1+L(g))^{n}}dg=\infty. 
\end{eqnarray*}
Finally we have found a unit positive vector $\xi$, such that for all $d\geq 1$ we have $$\int_{G}\frac{ \left\langle \pi(g) \xi,\xi\right\rangle ^{2}}{(1+L(g))^{d}}dg=\infty.$$ 
\end{proof}

\subsection{Proof of Theorem \ref{main}}
We state a very useful lemma due to Y. Shalom in \cite[Lemma 2.3]{S}:
\begin{lemma}\label{Shalom}
Let $\pi: G\rightarrow U(\mathcal{H})$ be a unitary representation. Assume that there exists a non-zero positive vector $\xi$ in the following sense: $$\langle \pi(g) \xi,\xi\rangle \geq 0,\forall g \in G .$$
Then for any  bounded measure positive measure $\mu$ on $G$ we have $$\|\lambda(\mu)\|_{op}\leq \|\pi(\mu)\|_{op}.$$
\end{lemma}

We prove Theorem \ref{main}.

\begin{proof}
According to Proposition \ref{propositiv}, it is enough to show that $\pi_{\alpha}$ does not satisfy property RD with respect to $L$. Suppose $\pi_{\alpha}$ has property RD with respect to $L$; we will find a contradiction. Let $\mathbb{Z}\subset \Gamma$ be a subgroup of exponential growth with respect to $L$. The restriction $\pi_{\mathbb{Z}}$ of $\pi_{\alpha}$ to $\mathbb{Z}$ would satisfy property RD with respect to $L$. Hence, according to Lemma \ref{Shalom} the left regular representation $\lambda_{\mathbb{Z}}$ would have property RD with respect to $L$ because of Lemma \ref{vecteurspositifs}. This is impossible because an amenable group can satisfy property RD with respect to $L$ if and only if it is a group of polynomial growth with respect to $L$ (see \cite[Proposition B, (2)]{Jo} ).  
 
\end{proof}

\subsection{Examples of representations associated to group actions without non zero invariant vectors}\label{invariant}
\subsubsection{Measure preserving ergodic group actions on infinite measured space}
Let $G$ be locally compact group group acting ergodically on a $\sigma$-finite measured space via $\alpha:G \curvearrowright (X,m)$. If $\pi_{\alpha}:G\rightarrow U\left(L^{2}(X,m)\right)$ has an invariant vector $\xi$, there exists a measurable function $\widetilde{\xi}$ with $\int_{X}\widetilde{\xi}^{2}(x)dm(x)<\infty$, which is $G$-invariant and satisfying $\widetilde{\xi}(x)=\xi(x)$ almost everywhere with respect to $m$ (see \cite[2.2.16 Lemma]{Zi}). Since the action is ergodic, $\widetilde{\xi}$ is a constant function on a conull set. Hence, if $m$ is an infinite $G$-invariant measure then $\pi_{\alpha}:G\rightarrow U\left(L^{2}(X,m)\right)$ does not contain non zero invariant vectors. 

\subsubsection{Nonsingular ergodic group actions of type $III_{1}$}
Let $\Gamma$ be a discrete countable group. A nonsingular action $\alpha:\Gamma\curvearrowright (X,m)$ of $\Gamma$ is an action of $\Gamma$ on the measured space $(X,m)$ such that $m$ is a $\Gamma$-quasi-invariant measure and there is no $\Gamma$-invariant measures in the measure class of $m$.
In the context of operator algebras we say that the equivalence relation produced by  $\Gamma\curvearrowright (X,m)$ is of type $III$. There exist different types of equivalence relation of type $III$ (see \cite[Chapter XIII, \S 2]{Ta}). We consider only the type $III_{1}$ case. 

Consider the Maharam extension (see \cite{Ma}) $\Gamma\curvearrowright (X\times\mathbb{R},dm \otimes e^{-t}dt)$ where $\Gamma$ acts by measure preserving transformations in the following way $$\gamma\cdot(x,t)=\left(\gamma\cdot x,t+\log \left(\frac{d{\gamma^{-1}_ {*}m}}{dm}(x)\right)\right).$$ We denote by $$\rho_{\alpha}:\Gamma \rightarrow U(L^{2}(X\times\mathbb{R},dm \otimes e^{-t}dt)),$$ the unitary representation associated to the Maharam extension. It is well known that if $\Gamma\curvearrowright (X,m)$ is of type $III_{1}$ then the action $\Gamma\curvearrowright (X\times\mathbb{R},dm \otimes e^{-t}dt)$ is ergodic. See \cite[Section 2]{Ko} for a survey about the type $III$ case. We have 

\begin{prop}
If $\alpha:\Gamma\curvearrowright (X,m)$ is of type $III_{1}$ then $\pi_{\alpha}$ does not have non zero invariant vectors.
\end{prop}
 \begin{proof}
Assume that $\pi_{\alpha}$ has an invariant vector. There exists a measurable function $\xi$ such that for all $\gamma$ in $\Gamma$, $\pi_{\alpha}(\gamma)\xi(x)=\xi(x)$ almost everywhere with respect to $m$. We shall prove that $\xi$ represents the null vector. Consider the function $$F(x,t)=\xi(x)e^{\frac{t}{2}}.$$ Observe that $\rho_{\alpha}(\gamma)F(x,t)=F(x,t)$ almost everywhere with respect to $dm\otimes e^{-t}dt$. There exists a measurable function $\widetilde{F}$ satisfying $\widetilde{F}=F$  almost everywhere with respect to $dm\otimes e^{-t}dt$, such that $\widetilde{F}$ is a $\Gamma$-invariant function (see \cite[2.2.16 Lemma]{Zi}). Since the action  $\Gamma\curvearrowright (X\times\mathbb{R},dm \otimes e^{-t}dt)$ is ergodic, $\widetilde{F}$ is a constant function on a conull set of $X\times\mathbb{R}$. Hence $\xi$ is a constant function on a conull set of $X$. Therefore $\xi=0$ almost everywhere with respect to $m$ (otherwise $\frac{d\gamma_{*}m}{dm}(x)=1$ almost everywhere which is excluded by hypothesis). 
\end{proof}

\subsubsection{Quasi-regular representation associated to a closed subgroup of a simple Lie group}

\begin{prop}\label{qrinvariant}
Let $G$ be a connected non-compact simple Lie group. Let $H$ be a closed subgroup of $G$ such that $G/H$ carries no finite $G$-invariant measure. Let $\Gamma$ be a lattice of $G$. Then the unitary representation  ${\lambda_{G/H}}_{|_{\Gamma}}:\Gamma \rightarrow U\left(L^{2}(G/H)\right)$ does not contain non zero invariant vectors.
\end{prop}
\begin{proof}
Consider the quasi-regular representation $\lambda_{G/H}:G \rightarrow U\left(L^{2}(G/H)\right)$ associated to $H$. The representation $\lambda_{G/H}$ can be identified with the induced representation  $Ind^{G}_{H}(1_{H})$ of $G$ associated to the trivial representation of $H$, see \cite[Example E.1.8 (ii)]{Be}. It is well known that $Ind^{G}_{H}(1_{H})$ has a non zero invariant vector if and only if $G/H$ carries a finite invariant regular Borel measure (\cite[Theorem E.3.1]{Be}). Hence $\lambda_{G/H}$ does not contain non zero invariant vectors. Combining this observation with Moore's Theorem \cite[2.2.19 Theorem]{Zi} we obtain that ${\lambda_{G/H}}_{|_{\Gamma}}$ does not contain non zero invariant vectors.
\end{proof}

\subsection{Proof of Corollary \ref{coro1}}

We give a proof of Corollary \ref{coro1}:
\begin{proof}
It is known that $\Gamma$ as in the corollary contains a cyclic subgroup with exponential growth with respect to any left-invariant Riemannian metric on $G$. See \cite[Theorem A]{Lu} for the higher rank case. The rank 1 one case is well known: horospherical subgroups of $\Gamma$ have exponential growth with respect to any left-invariant Riemannian metric, see \cite[Chapter 3, Section 3.C]{Gr} and \cite[Proposition 8. 25, p. 275 ]{BH}. So the result
follows from Theorem \ref{maintheo}, Proposition \ref{qrinvariant} ensuring that we are not in the
trivial case of Theorem \ref{maintheo}.
\end{proof}

\subsection{Proof of Corollary \ref{coro2}}

\begin{proof}
Let $G$ be a non compact semisimple Lie group endowed with a length function $L$ associated to a left-invariant Riemannian metric. Let $H$ be a closed subgroup of $G$. Let $\lambda_{G/H}:G \rightarrow U\left(L^{2}(G/H)\right)$ be the quasi-regular representation associated to $H$.\\
Let $u\in G$ be a unipotent element of infinite order. Let $\mathbb{Z}$ be the cyclic subgroup generated by $u$. Observe that $\mathbb{Z}$ is discrete and is an amenable group with exponential growth with respect to $L$. Thus it can not satisfy property RD with respect to $L$ by \cite[Proposition B, (2)]{Jo}. Since ${\lambda_{G/H}}_{|_{\mathbb{Z}}}$ has a positive vector (e.g. the characteristic function
of any subset of positive, finite measure) according to Lemma \ref{Shalom}, the representation ${\lambda_{G/H}}_{|_{\mathbb{Z}}}$ does not satisfy property RD with respect to $L$. Thanks to Proposition \ref{propositiv}  (2) applied to ${\lambda_{G/H}}_{|_{\mathbb{Z}}}$, there exists $\xi \in L^{2}(G/H)_{+}$ such that for all $d\geq 1$ we have $$\sum_{\gamma \in \mathbb{Z}}\frac{\left\langle \lambda_{G/H}(\gamma)\xi,\xi\right\rangle^{2}}{(1+L(\gamma))^{d}}=\infty.$$  
We claim that the weak inequality fails for the coefficient $\left\langle \lambda_{G/H}(g)\xi,\xi \right\rangle$. Assume on the contrary that it holds. There would exist $C_{\xi}$ and $d_{\xi}$ such that for all $g\in G $ $\left\langle \lambda_{G/H}(g) \xi,\xi\right\rangle \leq C_{\xi}(1+L(g))^{d_{\xi}}\Xi(g)$. For any $d_{0}>0$, we would have:
\begin{align*}
\sum_{\gamma \in \mathbb{Z}}\frac{\left\langle \lambda_{G/H}(\gamma)\xi,\xi\right\rangle^{2}}{(1+L(\gamma))^{2d_{\xi}+d_{0}}}& \leq C_{\xi}^{2} \sum_{\gamma \in \mathbb{Z}}\frac{(1+L(\gamma))^{2d_{\xi}}\Xi(\gamma)^{2}}{(1+L(\gamma))^{2d_{\xi}+d_{0}}}\\
&= C_{\xi}^{2} \sum_{\gamma \in \mathbb{Z}}\frac{\Xi(\gamma)^{2}}{(1+L(\gamma))^{d_{0}}} \\
&\leq C^{2}_{\xi}C \int_{G}\frac{\Xi(g)^{2}}{(1+L(g))^{d_{0}}}dg,
\end{align*}
where the last inequality follows from Proposition \ref{hcdiscret}.

Since $G$ is a semisimple Lie group, we can find $d_{0}>1$ such that $$\int_{G}\frac{\Xi^{2}(g)}{(1+L(g))^{d_{0}}}dg<\infty , $$ see
\cite[Chap 4, Theorem 4.6.4, p.161]{GV} for a reference. It would follow that $$\sum_{\gamma \in \mathbb{Z}}\frac{\left\langle \lambda_{G/H}(\gamma)\xi,\xi\right\rangle^{2}}{(1+L(\gamma))^{2d_{\xi}+d_{0}}}< \infty. $$ This is a contradiction.   
\end{proof}


\begin{thebibliography}{99}

  
   \bibitem{Be}B. Bekka, P. de la Harpe, and A. Valette, \emph{Kazhdan's property (T)}, New
Mathematical Monographs, vol. 11, Cambridge University Press, Cambridge, 2008.
	\bibitem{Boy}
	A. Boyer, \emph{Semisimple Lie groups satisfy property RD, a short proof} (English, French summary) 
C. R. Math. Acad. Sci. Paris 351 (2013), no. 9-10, 335�338.
	\bibitem{BH}	
M.R. Bridson, A. Haefliger, \emph{Metric spaces of non-positive curvature.} Grundlehren der Mathematischen Wissenschaften [Fundamental Principles of Mathematical Sciences], 319. Springer-Verlag, Berlin, 1999. 
	\bibitem{C}
	I. Chatterji, C. Pittet, and L. Saloff-Coste, \emph{Connected Lie groups and property RD},
Duke Math. J. 137 (2007), no. 3, 511-536.
	\bibitem{D}
	J. Dixmier, \emph{Les alg$\grave{e}$bres d'op$\acute{e}$rateurs dans l'espace hilbertien}, Gauthiers-Villars, Paris, 1964.
\bibitem{Fo}	
	 G-B. Folland. \emph{A course in abstract harmonic analysis} CRC
Press, 1995.
\bibitem{Gr}
 M. Gromov, \emph{Asymptotic invariants of infinite groups.} Geometric group theory, Vol. 2 (Sussex, 1991), 1-295, London Math. Soc. Lecture Note Ser., 182, Cambridge Univ. Press, Cambridge, 1993.
\bibitem{GV}
	R. Gangolli, V.S. Varadarajan, \emph{Harmonic Analysis of Spherical Functions on Real Reductive Groups}, Springer-Verlag, New-York, 1988.
	\bibitem{H}
	U. Haagerup, \emph{An example of a nonnuclear $C^*$-algebra which has the metric approximation property}, Invent. Math. 50 (1978/79), no. 3, 279-293.
\bibitem{He}
C. Herz, \emph{Sur le ph$\acute{e}$nom$\grave{e}$ne de Kunze-Stein}, C. R. Acad. Sci. Paris S�r. A-B 271 (1970), A491-A493.
	\bibitem{Jo}
	P. Jolissaint, \emph{Rapidly decreasing functions in reduced $C^*$-algebras of groups}, Trans. Amer. Math. Soc. 317 (1990), no. 1, 167-196.
\bibitem{Ko}
 H.Kosaki \emph{Type III factors and index theory.} Lecture Notes Series, 43. Seoul National University, Research Institute of Mathematics, Global Analysis Research Center, Seoul, 1998.
\bibitem{L}
	V. Lafforgue, \emph{K-th$\acute{e}$orie bivariante pour les $alg\grave{e}bres$ de Banach et conjecture de Baum-Connes}, Invent. Math. 149 (2002), no. 1, 1-95 (French).
	\bibitem{Lu}
	A. Lubotzky, S. Mozes, M. S. Raghunathan,  \emph{Cyclic subgroups of exponential growth and metrics on discrete groups.}, C. R. Acad. Sci. Paris S�r. I Math. 317 (1993), no. 8, 735�740.
	\bibitem{Ma}
	 D. Maharam, \emph{Incompressible transformations.} Fund. Math. 56 1964 35�50.

	\bibitem{Pa}
	T-W. Palmer, \emph{ Banach Algebras and the General Theory of ${}^*$-Algebras: Volume 2}, Encyclopedia of Mathematics and its Applications, vol. 49, Cambridge University Press, Cambridge, (2001)
	
	\bibitem{Re}
	H. Reiter. \emph{Classical harmonic analysis and locally compact
groups.} Oxford University Press, 1968.
	\bibitem{RS}
	M. Reed, B. Simon, \emph{I: Functional Analysis, Volume 1 (Methods of Modern Mathematical Physics)}, Academic press, (1980)
	\bibitem{R} 
	R-A. Ryan, \emph{Introduction to tensor product of Banach spaces},  Springer Monographs in Mathematics, (2002)
	\bibitem{S}
	Y. Shalom, \emph{Rigidity, unitary representations of semisimple groups, and fundamental groups of manifolds with rank one transformation group}, Ann. of Math. (2) 152
(2000), no. 1, 113-182.
  \bibitem{Sj}
	 P. Sj\"ogren, \emph{Convergence for the square root of the Poisson kernel.} Pacific J. Math. 131 (1988), no. 2, 361�391.
	\bibitem{Ta}
	M. Takesaki \emph{Theory of operator algebras. III.} Encyclopaedia of Mathematical Sciences, 127. Operator Algebras and Non-commutative Geometry, 8. Springer-Verlag, Berlin, 2003.
	\bibitem{Zi}
	 R-J. Zimmer, \emph{Ergodic theory and semisimple groups} Monographs in Mathematics, 81. Birkh�user Verlag, Basel, 1984.
\end{thebibliography}
\end{document}